\documentclass[a4paper,11pt]{amsart}

\usepackage[left=2.2cm,right=2.2cm,top=3.5cm,bottom=3cm]{geometry}

\usepackage{amsmath,amssymb,amscd,amsfonts}
\usepackage[hyphenbreaks]{breakurl}
\usepackage{mathrsfs}
\usepackage{latexsym}
\usepackage[all]{xy}
\usepackage{verbatim}
\usepackage[usenames, dvipsnames]{color}
\usepackage{multirow}
\usepackage{url}
\usepackage{mathdots}
\usepackage{MnSymbol}
\usepackage{stmaryrd}
\allowdisplaybreaks

\usepackage{hyperref}
\hypersetup{colorlinks,breaklinks,
           linkcolor=MidnightBlue,urlcolor=MidnightBlue,
                       anchorcolor=MidnightBlue,citecolor=MidnightBlue}

\newcommand{\N}{\mathbb{N}}
\newcommand{\Z}{\mathbb{Z}}
\newcommand{\Q}{\mathbb{Q}}
\newcommand{\C}{\mathbb{C}}
\newcommand{\F}{\mathbb{F}}
\renewcommand{\P}{\mathbb{P}} 
\newcommand{\G}{\Gamma} 
\newcommand{\T}{\mathbf{T}} 
\newcommand{\U}{\mathbf{U}}
\newcommand{\Tr}{\mathcal{T}} 
\renewcommand{\c}{\mathbf{c}} 
\newcommand{\I}{\mathcal{I}} 
\newcommand{\g}{\gamma}
\renewcommand{\a}{\alpha}
\renewcommand{\b}{\beta}
\renewcommand{\l}{\ell}
\newcommand{\e}{\overline{e}}

\renewcommand{\O}{\Omega}
\renewcommand{\o}{\omega}
\newcommand{\Hom}{\mathrm{Hom}}
\newcommand{\mf}{\,|_{k,m}}

\newtheorem{thm}{Theorem}[section]
\newtheorem{prop}[thm]{Proposition}
\newtheorem{lem}[thm]{Lemma}
\newtheorem{cor}[thm]{Corollary}
\newtheorem{defin}[thm]{Definition}
\newtheorem{rem}[thm]{Remark}

\newtheorem{exe}[thm]{Example}
\newtheorem{conj}[thm]{Conjecture}
\newtheorem{conjs}[thm]{Conjectures}

\renewcommand{\matrix}[4]{\left(\begin{array}{cc} {#1} & {#2} \\ {#3} & {#4} \end{array}\right)}
\newcommand{\smatrix}[4]{\left(\begin{smallmatrix} {#1} & {#2} \\ {#3} & {#4} \end{smallmatrix}\right)}

\title[Slopes of Drinfeld cusp forms]
{On the structure and slopes of Drinfeld cusp forms}

\author[Andrea Bandini]{Andrea Bandini$^*$}

\author[Maria Valentino]{Maria Valentino$^\dag$}\thanks{M. Valentino has been supported  by an ``Ing. G. Schirillo'' fellowship of INdAM}

\date{%
    $^*$Univesit\`a di Pisa, Dipartimento di Matematica, Largo Bruno Pontecorvo 5, 56127 Pisa, Italy \\%
    $^\dag$Universit\`a degli Studi di Parma, Dipartimento di Scienze Matematiche Fisiche e Informatiche, Parco Area delle Scienze, 53/A, 43124 Parma, Italy%
}

\begin{document}

\begin{abstract} We define oldforms and newforms for Drinfeld  cusp forms of level $t$ and conjecture
that their direct sum is the whole space of cusp forms. Moreover we describe explicitly the matrix $U$ associated to the 
action of the Atkin operator $\U_t$ on cusp forms of level $t$ and use it to compute tables of slopes of eigenforms.
Building on such data, we formulate conjectures on bounds for slopes, on the diagonalizability of $\U_t$ and on various 
other issues. Via the explicit form of the matrix $U$ we are then able to verify our conjectures in various cases
(mainly in small weights).
\end{abstract}

\maketitle

\section{introduction}

Let $N,k\in\Z_{\geqslant 0}$ and denote by $S_k(N)$ the $\C$-vector space of cuspidal modular forms of level $N$ and weight $k$.
Hecke operators $T_n$, $n\geqslant 1$, are defined on $S_k(N)$ and when a prime $p\in \Z$ divides the level $N$, $T_p$ is also known 
as the {\em Atkin}, or {\em Atkin-Lehner $U_p$-operator}.

A major topic in number theory is the construction of {\em families} of modular/cuspidal  forms and there are a number of related questions and conjectures about the 
{\em slopes} of such functions (e.g. bounds and recurring patterns for slopes). We recall that the $p$-slope of an eigenform, i.e. of a simultaneous eigenvector for all Hecke operators, is 
defined to be the $p$-adic valuation of its $U_p$-eigenvalue; in particular, an eigenform of $p$-slope zero is called {\em $p$-ordinary}.\\
The pioneer for the subject was Serre who, after developing the notion of a $p$-adic modular form, in \cite{Se1} presented the first $
p$-adic analytic family of modular eigenforms: the family of $p$-adic {\em Eisenstein series}.\\
A step further was then moved by Hida who provided a larger class of families of modular forms in the paper \cite{H2} and also studied $p$-adic analytic families of Galois representation 
attached to ordinary modular eigenforms in \cite{H1}.\\
Finally, we have the work of Coleman on {\em overconvergent} modular forms \cite{Co}: by proving that overconvergent modular forms of small slope 
(note that Coleman removed the restriction on ordinary eigenforms used by Hida) are classical, he found plenty
of $p$-adic families of classical modular forms. \\
In order to complete the picture, let us mention the article \cite{GM1} by Gouv\^ea and Mazur. In this paper, based on extensive numerical evidences, 
they asked some questions and stated a variety of conjectures on slopes and on the existence of families of modular forms.  In particular, 
they conjectured the generalization of Hida's theory to modular eigenforms of finite slope. It was this work that inspired
Coleman and motivated his search for the overconvergent families. However, we must mention that a counterexample to 
\cite[Conjecture 1]{GM1} was found by Buzzard and Calegari in \cite{BC}.

The interest of researchers for $U_p$ eigenvalues is not limited to slopes and families only: many related questions about diagonalizability
of $U_p$ and about the structure of $S_k(N)$ have been studied through the years and are well known in the case of number fields but, 
to our knowledge, have no counterpart yet in the function field setting. For example, regarding the diagonalizability of Hecke operators,
when $p$ is a prime number not dividing $N$, the action of all $T_p$ is semisimple on cusp forms.
This is no longer true for the action of $U_p$ on $S_k(Np)$, which fails to be diagonalizable. 
Some results on its semisimplicity are obtained in \cite{CE}. Moreover, the space of cusp forms of level $N$ and weight $k$ is direct sum 
of {\em newforms} and {\em oldforms}, which are mutually orthogonal with respect the Petersson inner product. It has been proved  
(see \cite{GM1}) that, for a fixed prime $p$, eigenvalues of old eigenforms have $p$-slope less than or equal to $k-1$; while new eigenforms 
all have slopes equal to $1-\frac{k}{2}$. Both results rely on Petersson inner product, a tool which is no longer avaliable for function
fields of characteristic $p$.\\

The present paper deals with the function field counterpart of (some of) the results mentioned above. The theory of modular forms for function  fields began with Drinfeld, indeed they were named
Drinfeld modular forms after him, but basic notions and definitions were actually introduced
only in the eighties by Goss and Gekeler (see, e.g. \cite{G1}, \cite{G2}, \cite{Go1} and \cite{Go2}).\\
For the sake of completeness we point out that Drinfeld modular forms are only half of the
story. In the realm of function fields there is another translation of classical modular forms:
the so called automorphic forms. They are functions on adelic groups with values in fields of
characteristic zero, but for the purpose of this paper we are only going to consider Drinfeld
modular forms which have values in a field of positive characteristic $p$.

Let $K=\F_q(t)$, where $q$ is a power of a fixed prime $p\in\Z$, and denote by
$A:=\F_q[t]$ its ring of integers (with respect to the prime at infinity $\frac{1}{t}$). Let $K_\infty$ be the completion of $K$ at 
$\infty:=\frac{1}{t}$ with ring of integers $A_\infty$ and denote by $\C_\infty$ the completion of an algebraic closure of $K_\infty$.\\
The finite dimensional $\C_\infty$-vector space of Drinfeld modular forms (more details on the objects mentioned in this introduction are in 
Section \ref{SecSet}) of weight $k\geqslant 0$ and type $m\in \Z$ for a congruence subgroup $\G<GL_2(A)$ is denoted by $M_{k,m}(\G)$. 
The corresponding space of cusp forms is indicated by $S^1_{k,m}(\G)$. If $\Gamma$ is the full $GL_2(A)$, in analogy with the classical case, 
we will refer to the related space of Drinfeld forms as forms of {\em level one}. 

In two recent papers \cite{BV1} and \cite{BV2}, we studied an analogue of the classical Atkin $U_p$-operator for a prime (hence any prime) of degree 1, i.e. the operator
$\U_t$, acting on the spaces $S^1_{k,m}(\Gamma_1(t))$ and $S^1_{k,m}(\Gamma_0(t))$.
In particular, we found an explicit formula for the action of $\U_t$ on $S^1_{k,m}(\Gamma_1(t))$, and here we shall use the 
matrix arising from that formula to deal with various issues like the structure of cusp forms of level $t$, diagonalizability and slopes of $\U_t$. We also started a computational search 
on eigenvalues and $t$-slopes of Atkin operators, looking for regularities and patterns in the distribution of $t$-slopes. 
The outcome of these computations are collected in tables  that can be downloaded from 
https://sites.google.com/site/mariavalentino84/publications.
Building on such tables and various other data, in the final section of \cite{BV2} we formulated some conjectures on slopes (e.g.
an analogue of Gouv\^ea-Mazur conjecture, see \cite[Conjecture 5.1]{BV2}) and on related issues. The aim of the present work is to 
explain these conjectures and their relations with structural issues of cusp forms spaces and also to give proofs in some special cases.\\
Regarding families in our setting, we would like to mention that Eisenstein series in positive characteristic were used by C. Vincent, together 
with trace and norm maps, to build examples of $\mathfrak{p}$-adic modular forms ($\mathfrak{p}$ a prime of $A$), i.e. families
of Drinfeld modular forms whose coefficients converge $\mathfrak{p}$-adically, see \cite[Definition 2.5 and Theorem 4.1]{Vi}. 
The analogue of Serre's families has been constructed by D. Goss in \cite{Go3}, using the $A$-expansion described by A. Petrov in \cite{Pe}. 
Moreover,  other progresses in the construction of more general 
families of Drinfeld modular forms have recently been achieved by S. Hattori in \cite{Ha2}, using the matrices we provide in 
Sections \ref{SecSymmetry} and \ref{SecMatrices} and building on his previous results on the analogue of Gouv\^ea-Mazur
Conjecture in \cite{Ha}, and by M.-H. Nicole and G. Rosso in \cite{NR}, employing a more geometric approach.\\

The main issues treated here are the following.
\begin{enumerate}
\item \underline{Injectivity of the Hecke operator}. We believe that the Hecke operator 
$\T_t$ acting on the space of cusp forms $S^1_{k,m}(GL_2(A))$ is injective for any weight $k$.

\item \underline{Diagonalizability of Hecke operators}. Inseparable eigenvalues occur both in level 1 and in level $t$
leading to non diagonalizable operators. Anyway we believe there is a more structural motivation for this (i.e.
the antidiagonal action of $\U_t$ on newforms, see Section \ref{SecAntidiagonalNewforms}), which causes
non diagonalizability in even characteristic only.

\item \underline{Newforms and Oldforms}. In \cite{BV2}, we defined two degeneracy maps 
$\delta_1,\delta_t: S^1_{k,m}(GL_2(A))\to S^1_{k,m}(\Gamma_0(t))$ and two trace 
maps (the other way around) to describe the subspaces of $S^1_{k,m}(\Gamma_0(t))$  composed by newforms and oldforms
denoted, respectively, by $S^{1,new}_{k,m}(\Gamma_0(t))$ and $S^{1,old}_{k,m}(\Gamma_0(t))$. We believe these
definition will provide a decomposition of $S^1_{k,m}(\Gamma_0(t))$ as the direct sum of oldforms and newforms. Section \ref{SecSpecial} will provide evidence for the conjecture and a computational criterion for it.

\item \underline{Bounds on slopes}. It is easy to find a lower bound for slopes (see Proposition \ref{SmallestSlope}), we believe 
an upper bound is $\frac{k}{2}$ (such bound will have consequences also on the previous issues, see Remark \ref{RemTF} for details). The current upper bound of Theorem 
\ref{ThmUpperBSlopes} is unfortunately still quite far from it. 

\end{enumerate}

The paper is organized as follows. 
In Section \ref{SecSet}, we fix notations  and recall the main results from \cite{BV2} that we are going to use throughout the paper.
Moreover, we formulate the conjectures we shall work on in the subsequent sections.

\begin{conjs}[Conjecture \ref{OldConj}]\
\begin{enumerate}
\item[{\bf 1.}] $Ker(\T_t)=0$;
\item[{\bf 2.}] $\U_t$ is diagonalizable when $q$ is odd and, when $q$ is even, it is diagonalizable if and only if the dimension of
$S^{1,new}_{k,m}(\Gamma_0(t))$ is 1;
\item[{\bf 3.}] $S^1_{k,m}(\Gamma_0(t))=S^{1,old}_{k,m}(\Gamma_0(t))\oplus S^{1,new}_{k,m}(\Gamma_0(t))$.
\end{enumerate}
\end{conjs}

In Section \ref{SecSpBlocks} we describe a matrix $M$ associated to the action of $\U_t$ on $S^1_{k,m}(\Gamma_0(t))$.
In Section \ref{SecMatrices} we translate all Conjectures \ref{OldConj} in linear algebra problems thanks to the previous calculation on $M$. 
In Section \ref{SecSpecial} we use tool from  Section \ref{SecMatrices} to prove several cases of the conjectures (in particular for all weights
$k\leqslant 5(q-1)$) and to present an equivalent formulation of conjecture {\bf 3} above (see Theorem \ref{ThmSum}).
Finally, in Section \ref{SecBounds} using a Newton Polygon argument we give upper and lower bounds on slopes and on the dimension
of the space of fixed slope (i.e. on the number of independent eigenforms with a fixed slope), for which we find a result comparable with the one of K. Buzzard in \cite{Bu}
for the characteristic 0 case. \\

\noindent {\bf Akwnoledgements}: We would like to thank the anonymous referee for his/her prompt report and
for informing us of the ongoing work of G. B\"ockle, P. Graef and R. Perkins on Maeda's conjecture.

\section{Setting and notations}\label{SecSet}
Let $K$ be the global function field $\F_q(t)$, where $q$ is a power of a fixed prime $p\in\Z$,
 fix the prime $\frac{1}{t}$ at $\infty$ and denote by
$A:=\F_q[t]$ its ring of integers (i.e., the ring of functions regular outside $\infty$).
Let $K_\infty$ be the completion of $K$ at $\frac{1}{t}$ with ring of integers $A_\infty$
and denote by $\C_\infty$ the completion of an algebraic closure of $K_\infty$.\\
The {\em Drinfeld upper half-plane} is the set $\O:=\P^1(\C_\infty) - \P^1(K_\infty)$
together with a structure of rigid analytic space (see \cite{FvdP}).

\subsection{The Bruhat-Tits tree}\label{SecTree}
The Drinfeld upper half plane has a combinatorial counterpart, the {\em Bruhat-Tits tree} $\Tr$ of $GL_2(K_\infty)$, which
we shall describe briefly here. For more details the reader is referred to \cite{G3}, \cite{G4} and \cite{S1} (a short summary of the
relevant information is also provided in \cite[Section 2.1]{BV2}).\\
The tree $\Tr$ is a $(q+1)$-regular tree on which $GL_2(K_\infty)$ acts transitively. Let us denote by $Z(K_\infty)$ the scalar
matrices of $GL_2(K_\infty)$ and by $\I(K_\infty)$ the {\em Iwahori subgroup}, i.e.,
\[ \I(K_\infty)=\left\{\matrix{a}{b}{c}{d} \in GL_2(A_\infty)\, : \, c\equiv 0 \pmod \infty \right\}. \]
Then the sets $X(\Tr)$ of vertices and $Y(\Tr)$ of oriented edges of $\Tr$ are given by
\[ X(\Tr) = GL_2(K_\infty)/Z(K_\infty)GL_2(A_\infty)\ \ {\rm and}\ \
Y(\Tr) = GL_2(K_\infty)/ Z(K_\infty)\I(K_\infty). \]
The canonical map from $Y(\Tr)$ to $X(\Tr)$ associates with each oriented edge $e$ its origin $o(e)$
(the corresponding terminus will be denoted by $t(e)$). The edge $\overline{e}$ is $e$ with
reversed orientation.\\
Two infinite paths in $\Tr$ are considered equivalent if they differ at finitely many edges. An {\em end} is an equivalence
class of infinite paths. There is a $GL_2(K_\infty)$-equivariant bijection between the ends of $\Tr$ and $\P^1(K_\infty)$.
An end is called {\em rational} if it corresponds to an element in $\P^1(K)$ under the above bijection.
Moreover, for any arithmetic subgroup $\G$ of $GL_2(A)$, the elements of $\G\backslash \P^1(K)$ are in bijection with
the ends of $\G\backslash \Tr$ (see \cite[Proposition 3.19]{B} and \cite[Lecture 7, Proposition 3.2]{GPRV})
and they are called the {\em cusps} of $\G$.\\
Following Serre \cite[pag 132]{S1}, we call a vertex or an edge {\em $\G$-stable} if its stabilizer in $\G$ is trivial
and {\em $\G$-unstable} otherwise.

\subsection{Drinfeld modular forms}\label{SecDrinfModForms}
The group $GL_2(K_\infty)$ acts on $\Omega$ via M\"obius transformation
\[  \left( \begin{array}{cc}
a & b  \\
c & d
\end{array} \right)(z)= \frac{az+b}{cz+d}. \]
Let $\G$ be an arithmetic subgroup of $GL_2(A)$. It has finitely many cusps, represented by $\G\backslash \P^1(K)$.
For $\g =\smatrix{a}{b}{c}{d}\in GL_2(K_\infty)$, $k,m \in \Z$
and $\varphi:\O\to \C_\infty$, we define
\begin{equation}\label{Mod0}
(\varphi \mf \g)(z) := \varphi(\g z)(\det \g)^m(cz+d)^{-k}.
\end{equation}

\begin{defin}
A rigid analytic function $\varphi:\O\to \C_\infty$ is called a {\em Drinfeld modular function of weight $k$ and type $m$ for $\G$} if
\begin{equation}\label{Mod} (\varphi \mf \g )(z) =\varphi(z)\ \ \forall \g\in\G.  \end{equation}
A Drinfeld modular function $\varphi$ of weight $k\geqslant 0$ and type $m$ for $\G$ is called a {\em Drinfeld modular form} if
$\varphi$ is holomorphic at all cusps.\\
A Drinfeld modular form $\varphi$ is called a {\em cusp form} if  it vanishes at all
cusps.\\
The space of Drinfeld modular forms
of weight $k$ and type $m$ for $\G$ will be denoted by $M_{k,m}(\G)$. The subspace of cuspidal modular forms is denoted by $S^1_{k,m}(\G)$.
\end{defin}

The above definition coincides with \cite[Definition 5.1]{B}, other authors require the function to be 
meromorphic (in the sense of rigid analysis, see for example \cite[Definition 1.4]{Cor}) and would call our
functions {\em weakly modular}.

Weight and type are not independent of each other:
if $k\not\equiv 2m \pmod{o(\Gamma)}$, where $o(\Gamma)$ is the number of scalar matrices in $\Gamma$, then $M_{k,m}(\G)=0$.
Moreover, if all elements of $\G$ have determinant 1, then equation \eqref{Mod0} shows that
the type does not play any role. If this is the case, for fixed $k$ all $M_{k,m}(\G)$ are isomorphic (the same holds for
$S^1_{k,m}(\G)$ and we will simply denote them by $M_{k}(\G)$ (resp. $S^1_{k}(\G)$).

\noindent All $M_{k,m}(\G)$ and $S^1_{k,m}(\G)$ are finite dimensional $\C_\infty$-vector spaces. For details on 
the dimension of these spaces see \cite{G1}. \\
Since $M_{k,m}(\G) \cdot M_{k',m'}(\G)\subset M_{k+k',m+m'}(\G)$ we have that
\[  M(\G) = \bigoplus_{k,m} M_{k,m}(\G)\quad \mathrm{and}\quad M^0(\G)=\bigoplus_k M_{k,0}(\G) \]
are graded $\C_\infty$-algebras. \\
Moreover, let 
\[ g\in M_{q-1,0}(GL_2(A)),\quad \Delta\in S^1_{q^2-1,0}(GL_2(A))\quad \mathrm{and}\quad h\in M_{q+1,1}(GL_2(A)) \]
be as in \cite[Sections 5 and 6]{G2}, then
\begin{equation}\label{GradAlg}
M^0(GL_2(A))=\C_\infty[g,\Delta]\quad \mathrm{and}\quad M(GL_2(A))=\C_\infty[g,h]\,.
\end{equation}

\subsection{Harmonic cocycles}\label{SecHarCoc}
For $k> 0$ and $m\in\Z$, let $V(k,m)$ be the $(k-1)$-dimensional vector space over $\C_\infty$ with basis
$\{X^jY^{k-2-j}: 0\leqslant j\leqslant k-2 \}$. The action of $\g=\smatrix{a}{b}{c}{d} \in GL_2(K_\infty)$ on $V(k,m)$ is given by
\[ \g(X^jY^{k-2-j}) = \det(\g)^{m-1}(dX-bY)^j(-cX+aY)^{k-2-j}\quad {\rm for}\ 0\leqslant j\leqslant k-2.\]
For every $\o\in \Hom(V(k,m),\C_\infty)$ we have an induced action of $GL_2(K_\infty)$
\[ (\g\o)(X^jY^{k-2-j})=\det(\g)^{1-m}\o((aX+bY)^j(cX+dY)^{k-2-j})\quad {\rm for}\ 0\leqslant j\leqslant k-2. \]
\begin{defin}
A {\em harmonic cocycle of weight $k$ and type $m$ for $\G$} is a function $\c$ from the set of directed edges
of $\Tr$ to $\Hom(V(k,m),\C_\infty)$ satisfying:
\begin{itemize}
\item[{\bf 1.}] {\em (harmonicity)} for all vertices $v$ of $\Tr$,
$\displaystyle{\sum_{t(e)= v}\c(e)=0}$,
where $e$ runs over all edges in $\Tr$ with terminal vertex $v$;
\item[{\bf 2.}] {\em (antisymmetry)} for all edges $e$ of $\Tr$, $\c(\overline{e})=-\c(e)$;
\item[{\bf 3.}] {\em ($\G$-equivariancy)} for all edges $e$ and elements $\g\in\G$, $\c(\g e)=\g(\c(e))$.
\end{itemize}
\end{defin}

\noindent The space of harmonic cocycles of weight $k$ and type $m$ for $\G$ will be denoted by $C^{har}_{k,m}(\G)$.

\subsubsection{Cusp forms and harmonic cocycles}\label{SecIsomModFrmHarCoc}
In \cite{T}, Teitelbaum constructed the so-called ``residue map'' which allows us to interpret cusp forms as harmonic cocycles.
Indeed, it is proved in \cite[Theorem 16]{T} that this map is actually an isomorphism
$S^1_{k,m}(\G)\simeq C^{har}_{k,m}(\G)$.\\
For more details the reader is referred to the original paper of Teitelbaum \cite{T} or to
\cite[Section 5.2]{B}, where the author gives full details in a more modern language. We remark that the two papers have different 
normalizations (as mentioned in \cite[Remark 5.8]{B}): here we adopt Teitelbaum's one but, working as in \cite[Section 5.2]{B} 
where computations for the residue map are detailed right after \cite[Definition 5.9]{B}, we obtain \cite[equation (17)]{B} which carries 
the action of the Hecke operators on harmonic cocycles (see next section).

\subsection{Hecke operators}\label{SecHecke}
\noindent We shall focus on the congruence groups $\G:=\G_0(t), \ \G_1(t)$ defined as
\[ \G_0(t)=\left\{ \left( \begin{array}{cc}
a & b  \\
c & d
\end{array} \right)\in GL_2(A): c\equiv 0 \pmod{t} \right\}\]
and
\[ \G_1(t)=\left\{ \left( \begin{array}{cc}
a & b  \\
c & d
\end{array} \right)\in GL_2(A): a\equiv d\equiv 1\ \mathrm{and}\ c\equiv 0 \pmod{t} \right\}.\]

If $\varphi\in M_{k,m}(GL_2(A))$ the Hecke operator is defined in the following way
\begin{align*} \T_t(\varphi)(z) & : = t^{k-m} ( \varphi \mf \matrix{t}{0}{0}{1} ) (z) +
t^{k-m} \sum_{b\in \F_q}  ( \varphi \mf \matrix{1}{b}{0}{t} ) (z)\\
 & = t^k\varphi(tz)+ \sum_{b\in\F_q}\varphi\left( \frac{z+b}{t}\right). \end{align*}
While, if $\varphi\in M_{k,m}(\G)$, $\G=\G_1(t)$, $\G_0(t)$, we have the analogue of the Atkin-Lehner operator
\begin{align*} \U_t(\varphi)(z) & :=t^{k-m}\sum_{b\in \F_q}  ( \varphi \mf \matrix{1}{b}{0}{t} ) (z) = \sum_{b\in\F_q}\varphi\left( \frac{z+b}{t}\right). \end{align*}

\subsection{Action of $\U_t$ on $\G_1(t)$-invariant cusp forms}
In order to describe the action of $\U_t$ on $S^1_k(\G_1(t))$ it is convenient to exploit the harmonic cocycles description of them.\\
The residue map allows us to define a Hecke action on harmonic cocycles in the following way:
\begin{align*}
\U_t(\c(e))= t^{k-m} \sum_{b\in \F_q}  \left(\begin{array}{cc}
1 & b  \\
0 & t
\end{array}\right)^{-1}\c\left( \left(\begin{array}{cc}
1 & b  \\
0 & t
\end{array}\right)e\right)
\end{align*}
(for details see formula (17) in \cite[Section 5.2]{B}, recalling Section \ref{SecIsomModFrmHarCoc}). \\
By \cite[Proposition 5.4]{B} and \cite[Corollary 5.7]{GN} we have that
\[  \dim_{\C_\infty}S^1_k(\G_1(t))=k-1\,. \]
Moreover, as a consequence of \cite[Lemma 20]{T}, cocycles in $C^{har}_{k,m}(\G_1(t))$ are determined by their
values on a stable edge $\bar{e}=(\begin{smallmatrix} 0 & 1\\ 1 & 0 \end{smallmatrix})$ of a fundamental domain 
for $\G_1(t)\backslash \Tr$ (the computations for fundamental domains are carried out in \cite{GN}, a short description of the
$\G_1(t)$ case is in \cite[Section 4]{BV2}). Therefore, for any $j\in\{0,1,\dots,k-2\}$, let $\c_j(\e)$ be defined by
\[ \c_j(\e)(X^iY^{k-2-i})=\left\{ \begin{array}{ll} 1 & {\rm if}\ i=j \\
0 & {\rm otherwise} \end{array} \right. \ .\]
The set $\mathcal{B}^1_k(\G_1(t)):=\{\c_j(\e),\,0\leqslant j\leqslant k-2\}$ is a basis for
$S^1_k(\G_1(t))$. By \cite[Section 4.2]{BV2} we have

\begin{align}\label{Ttcj}
\U_t(\c_j(\e)) & = -(-t)^{j+1} \binom{k-2-j}{j} \c_j(\e) -t^{j+1}\sum_{h\neq 0}\left[ \binom{k-2-j-h(q-1)}{-h(q-1)} \right.\\
\ &\left. + (-1)^{j+1} \binom{k-2-j-h(q-1)}{j} \right] \c_{j+h(q-1)}(\e)  \nonumber
\end{align}
(where it is understood that $\c_{j+h(q-1)}(\e) \equiv 0$ whenever $j+h(q-1)<0$ or $j+h(q-1)>k-2$).

From formula \eqref{Ttcj} one immediately notes that the $\c_j$ can be divided into classes modulo $q-1$ and
every such class is stable under the action of $\U_t$.
For any $0\leqslant j\leqslant q-2$, we shall denote by $C_j$ the class of $\c_j(\e)$, i.e, $C_j=\{\c_\ell(\e)\,:\,\ell\equiv{j}\pmod{q-1}\}$: the cardinality
of $C_j$ is the largest integer $n$ such that $j+(n-1)(q-1)\leqslant k-2$ (note that it is possible to have
$|C_j|=0$, exactly when $j>k-2$). 

\subsection{Newforms and oldforms}
We recall here our definitions of newforms and oldforms, and the main properties/formulas
for various maps between spaces of cusp forms (all details are in the paper \cite{BV2}).

\subsubsection{Oldforms}
Consider the injective map (see \cite[Proposition 3.1]{BV2})
\[ \delta : S_{k,m}^1(GL_2(A))\times S_{k,m}^1(GL_2(A)) \longrightarrow S_{k,m}^1(\G_0(t)) \]
\[ \delta(\varphi,\psi):=\delta_1\varphi+\delta_t\psi \]
where
\[ \delta_1,\delta_t: S^1_{k,m}(GL_2(A)) \rightarrow S^1_{k,m}(\G_0(t)) \]
\[ \delta_1(\varphi):=\varphi \]
\[ \delta_t(\varphi):=(\varphi\mf \matrix{t}{0}{0}{1})(z)\ {\rm ,i.e.,}\ (\delta_t(\varphi))(z)=t^m\varphi(tz). \]

\begin{defin}
{\em Oldforms of level $t$} are elements of $S^{1,old}_{k,m}(\Gamma_0(t)):=Im(\delta)$.
\end{defin}

Let $\varphi\in S^1_{k,m}(GL_2(A))$. We have that (see \cite[Section 3.2]{BV2}):
\begin{equation}\label{EqUtdelta1}
\delta_1(\T_t\varphi)= t^{k-m}\delta_t(\varphi)+ \U_t(\delta_1(\varphi))
\end{equation}
\begin{equation}\label{EqUtdeltat}
\U_t(\delta_t(\varphi))=0
\end{equation}

\subsubsection{Newforms}
Let
\[ \gamma_t:=\matrix{0}{-1}{t}{0} \]
be the {\em Fricke involution}.
To shorten notations we shall often use $\varphi^{Fr}$ to denote $(\varphi\mf \gamma_t)$. \\
It is easy to see that $(\varphi^{Fr})^{Fr}= t^{2m-k} \varphi$. Moreover, noting that $\matrix{0}{-1}{1}{0}\in GL_2(A)$ and that
$\matrix{0}{-1}{1}{0}\matrix{t}{0}{0}{1}=\gamma_t$, one readily observes that $\varphi^{Fr}=\delta_t(\varphi)$ 
for any $\varphi\in S^1_{k,m}(GL_2(A))$ (this final relation makes no sense for forms in $S^1_{k,m}(\Gamma_0(t))$ on which
$\delta_t$ is not defined, we also remark that $\varphi^{Fr}\neq (\varphi\mf \matrix{t}{0}{0}{1})$ in general).

To define the trace maps we use the following system of representatives for $GL_2(A)$ modulo $\Gamma_0(t)$:
\[ R:=\left\{ {\bf Id}_2, \matrix{0}{-1}{1}{b}\ b\in \F_q \right\}.\]

\begin{defin}\label{DefTrace}
For any cuspidal form $\varphi$ of level $t$ define the {\em trace}
\[ Tr(\varphi):=\sum_{\gamma\in R} (\varphi\mf \gamma) \]
and the {\em twisted trace}
\[ Tr'(\varphi):=Tr(\varphi^{Fr}) =\sum_{\gamma\in R} (\varphi\mf \gamma_t\gamma) .\]
Both $Tr$ and $Tr'$ are maps from $S_{k,m}^1(\G_0(t))$ to $S_{k,m}^1(\G_0(1))$ (see \cite[Definition 3.5]{Vi}).
\end{defin}

Let $\varphi\in S^1_{k,m}(\G_0(t))$. We have that (see \cite[Section 3.3]{BV2}):
\begin{equation}\label{EqTr}  Tr(\varphi)=\varphi+t^{-m}\U_t(\varphi^{Fr}) ,\end{equation}
\begin{equation}\label{EqTr'}   Tr'(\varphi)= \varphi^{Fr}+ t^{m-k}\U_t(\varphi) .\end{equation}
Moreover, for any $\varphi\in S^1_{k,m}(GL_2(A))$ (see \cite[Section 3.4]{BV2}), one has
\begin{equation}\label{EqTrdelta} 
Tr(\delta_1(\varphi))=\varphi \quad {\rm and}\quad Tr(\delta_t(\varphi)) =t^{m-k} \T_t\varphi .
\end{equation} 

Let $\varphi\in S^1_{k,m}(GL_2(A))$ be a $\T_t$-eigenform of eigenvalue $\lambda\neq 0$. Then 
$\delta(\varphi,-\frac{t^{k-m}}{\lambda}\varphi)\in S^1_{k,m}(\Gamma_0(t))$ is a $\U_t$-eigenform of eigenvalue $\lambda$.
One can actually prove that $\{Eigenvalues\ of\ {\U_t}_{|Im(\delta)}\}=\{Eigenvalues\ of\ \T_t\}\cup\{0\}$ 
(see \cite[Proposition 3.6]{BV2}, the $0$ comes from $Ker(\U_t)=Im(\delta_t)\,$), so we have information on ``old eigenvalues''.
Moreover one can check that $\delta(\varphi,-\frac{t^{k-m}}{\lambda}\varphi)\in Ker(Tr)$ for any $\varphi$ as above, hence the kernel of
the trace is not enough to distinguish newforms (as it was in the classical case, see e.g. \cite[Section 4]{GM1}).

\begin{defin}\label{DefNewforms}
{\em Newforms of level $t$} are elements in $S^{1,new}_{k,m}(\Gamma_0(t)):=Ker(Tr)\cap Ker(Tr')$.
\end{defin}

Let $\varphi$ be a newform of level $t$ which is also an $\U_t$ eigenform of eigenvalue $\lambda$. Then by \eqref{EqTr}
and \eqref{EqTr'} we have that $\varphi=-t^{-m}\U_t(\varphi^{Fr})$ and $\varphi^{Fr}=-t^{m-k}\U_t(\varphi)$. Hence
\[  \lambda^2\varphi=\U_t^2(\varphi)=t^k \varphi. \]
Then, newforms can only have eigenvalues $\pm t^{\frac{k}{2}}$ and slope $\frac{k}{2}$.

\subsection{Conjectures}
Numerical data (see also \cite[Section 5]{BV2}) and comparison with the classical case led us to the following conjectures

\begin{conjs}\label{OldConj}\ 
\begin{enumerate}
\item[{\bf 1.}] $Ker(\T_t)=0$;
\item[{\bf 2.}] $\U_t$ is diagonalizable when $q$ is odd and, when $q$ is even, it is diagonalizable if and only if the dimension of
$S^{1,new}_{k,m}(\Gamma_0(t))$ is 1;
\item[{\bf 3.}] $S^1_{k,m}(\Gamma_0(t))=S^{1,old}_{k,m}(\Gamma_0(t))\oplus S^{1,new}_{k,m}(\Gamma_0(t))= Im(\delta)\oplus(Ker(Tr)\cap Ker(Tr'))$.
\end{enumerate}
\end{conjs}

A few words on Conjecture {\bf 2}: we already have examples of non diagonalizability in even characteristic provided in \cite{BV1} and 
\cite[Section 5]{BV2} and they all seem to depend on the fact that the action of $\U_t$ on newforms has the tendency to being antidiagonal
(for more examples see Section \ref{SecAntidiagonalNewforms}). Such matrices (with only one eigenvalue, namely $t^{\frac{k}{2}}$ as mentioned before)
are never diagonalizable in even characteristic (unless, of course, they have dimension 1), hence our conjecture. Moreover it is easy to see
that, if $\T_t$ is diagonalizable on $S^1_{k,m}(GL_2(A))$, then $\U_t$ is diagonalizable on $Im(\delta)=S^{1,old}_{k,m}(\Gamma_0(t))$
if and only if $\T_t$ is injective. Therefore our Conjecture {\bf 1} can be seen an a first step towards (or, thanks to Conjecture {\bf 3}, as a 
consequence of) Conjecture {\bf 2}.

\begin{rem}
In the characteristic zero case Maeda's Conjecture \cite{HM} predicts that in level $N=1$ for a prime $p\in\Z$ the polynomial
\[  P_{k,p}(X):=\prod_{f} (X-a_p(f)) \]
where $f=\sum_n a_n(f)q^n\in S_k(SL_2(\Z))$ runs over all normalized eigenforms (i.e. such that $a_1(f)=1$) of a chosen basis, 
is irreducible over $\Q$. 
Moreover, Maeda conjectured that the Galois group of $P_{k,p}(X)$ is the full symmetric group $S_d$ where $d=\dim_\C S_k(SL_2(\Z))$.\\
It is clear from our tables that Maeda's conjecture has to be reformulated in even characteristic when there are inseparable eigenvalues.
Besides, even in odd characteristic eigenvalues are mostly in $\F_q[t]$, as one can see by the above mentioned tables,
 and this proves that an analogue of Maeda's conjecture is false in our setting \footnote{The anonymous referee kindly informed us  
 that there is some ongoing work by G. B\"oeckle, P. Graef and R. Perkins on suitable formulations of Maeda's conjecture in the Drinfeld setting.}.
\end{rem}

\section{The blocks associated to $S^1_{k,m}(\G_0(t))$}\label{SecSpBlocks}
A set of representatives for $\G_0(t)/\G_1(t)$ is provided by the $(q-1)^2$ matrices
$R^0_1=\left\{ \matrix{a}{0}{0}{d}\,:\,a,d\in \F_q^*\right\}$, hence a cocycle $\c_j$ comes from $S^1_{k,m}(\G_0(t))$
if and only if it is $R^0_1$-invariant. Direct computation leads to
\[ \matrix{a}{0}{0}{d}^{-1}\c_j\left(\matrix{a}{0}{0}{d}\e\right)(X^\ell Y^{k-2-\ell})=
a^{m-1-\ell}d^{m-k+\ell+1}\c_j(\e)(X^\ell Y^{k-2-\ell}).\]
Therefore
\[ \matrix{a}{0}{0}{d}\cdot\c_j=a^{m-1-j}d^{m-k+j+1}\c_j\quad \forall j \]
and this is $R^0_1$-invariant if and only if
\[ a^{m-1-j}d^{m-k+j+1}=1 \quad \forall a,d\in \F_q^*.\]
This yields
\[ j\equiv m-1\equiv k-m-1 \pmod{q-1} ,\ {\rm i.e.}\ k\equiv 2j+2 \pmod{q-1} \]
(and $k\equiv 2m\pmod{q-1}$ as natural to get a nonzero space of cuspidal forms for $\G_0(t)$). If $q$ is even this
provides a unique class $C_j$, if $q$ is odd then we have two solutions: $j$ (assumed to be the smallest nonnegative one)
and $j+\frac{q-1}{2}$. Note that in any case $k$ has the form $2j+2+(n-1)(q-1)=2(j+\frac{q-1}{2})+2+(n-2)(q-1)$ for some integer 
$n$ and the classes corresponding to $S^1_{k,m}(\G_0(t))$ are determined by the type $m$ (as predictable since $m$ plays a 
role in $S^1_{k,m}(\G_0(t))$ but not in $S^1_k(\G_1(t))$, because all matrices in $\G_1(t)$ have determinant 1). 
If $q$ is even then the unique class has dimension $n$, while for odd $q$ we have
\[ |C_j|=n\quad{\rm and}\quad |C_{j+\frac{q-1}{2}}|=n-1 .\]

\begin{rem}\label{RemDimCusp}
This could be seen as an easy alternative to the Riemann-Roch argument usually used to
compute the dimension of such spaces, see for example \cite[Section 4]{Cor}.
\end{rem}

\subsection{Matrices associated to $C_j$}\label{SecSymmetry} Since we will focus on the block(s)
coming from level $\G_0(t)$ only (unless stated otherwise), when we speak about the block $C_j$ of dimension $n$
we always imply that $j$ and $n$ are such that $k=2j+2+(n-1)(q-1)$ (formulas for $C_{j+\frac{q-1}{2}}$ are the same, 
just substitute $j$ with $j+\frac{q-1}{2}$ and take into account
the different parity of the dimension $n-1$).\\

\noindent Using formula \eqref{Ttcj} one finds that the general entries of the matrix associated to the action of $\U_t$ 
on $S^1_{k,j+1}(\G_1(t))$ are ($a,b$ are now the row-columns indices): 
\begin{equation}\label{EqCoeffMj}
m_{a,b}(j,k) = \left\{ \begin{array}{ll} \displaystyle{-t^{j+1+(b-1)(q-1)}\left[\binom{k-2-j-(a-1)(q-1)}{(b-a)(q-1)} \right.} & \\
\displaystyle{\left. +(-1)^{j+1+(b-1)(q-1)}\binom{k-2-j-(a-1)(q-1)}{j+(b-1)(q-1)}\right]} & {\rm if}\ a\neq b \\
\ & \\
\displaystyle{-(-t)^{j+1+(a-1)(q-1)}\binom{k-2-j-(a-1)(q-1)}{j+(a-1)(q-1)}} & {\rm if}\ a = b \end{array}\right.
\end{equation}
(for future reference note that for any $q$ one has $(-1)^{(\ell-1)(q-1)}=1$ for any $\ell$). Remember that $0\leqslant j\leqslant q-2$
and so the type is $0$ when $j=q-2$.

We denote by $M$ the coefficient matrix (i.e., the one without the powers of $t$) associated to the action of $\U_t$ on $C_j$.

Specializing formula \eqref{EqCoeffMj} at our particular value of $k$, we see that the general entries of $M$ are
\begin{equation}\label{spam}
\displaystyle{ m_{a,b}= \left\{\begin{array}{ll}
\displaystyle{-\left[\binom{j+(n-a)(q-1)}{j+(n-b)(q-1)} + (-1)^{j+1} \binom{j+(n-a)(q-1)}{j+(b-1)(q-1)}\right]} & {\rm if}\ a\neq b \\
\displaystyle{(-1)^j\binom{j+(n-a)(q-1)}{j+(a-1)(q-1)}} & {\rm if}\ a=b \end{array}\right. .}
\end{equation}
It is easy to check that $M$ satisfies some symmetry relations. We write down those for even $n$, the other case
is similar. In particular
\begin{itemize}\label{symmetry}
\item[S1.] {\em symmetry between columns}: $m_{a,n+1-b}=(-1)^{j+1}m_{a,b}$ for any $a\neq b, n+1-b$,
i.e., outside diagonal and antidiagonal (because of this we shall simply check the first $\frac{n}{2}$ columns from now on);
\item[S2.] {\em symmetry between diagonal and antidiagonal}: $m_{a,n+1-a}=(-1)^{j+1}(m_{a,a}-1)$ for any $a\neq n+1-a$;
\item[S3.] {\em antidiagonal, $\frac{n}{2}+1\leqslant a\leqslant n$}:
\[ m_{a,n+1-a}=-\left[\binom{j+(n-a)(q-1)}{j+(a-1)(q-1)} + (-1)^{j+1} \binom{j+(n-a)(q-1)}{j+(n-a)(q-1)}\right]=(-1)^j \]
(because in our range $n-a<a-1$). This yields
\[ (-1)^j=(-1)^{j+1}(m_{a,a}-1),\quad{\rm i.e.}\quad m_{a,a}=0 \]
in the range in which S2 and S3 hold.
\item[S4.] {\em below antidiagonal, $\frac{n}{2}+1\leqslant a\leqslant n-1$}:
\[ -\left[\binom{j+(n-a)(q-1)}{j+(n-b)(q-1)} + (-1)^{j+1} \binom{j+(n-a)(q-1)}{j+(b-1)(q-1)}\right]=0 \]
(because in our range $n-a<n-b$ and $n-a<b-1$);
\end{itemize}

Putting all these information together we can see that for any even $n$ the matrix $M$
has the following shape
{ \small\[ \left(\begin{array}{cccccccc} m_{1,1} & m_{1,2} & \cdots & m_{1,\frac{n}{2}} & (-1)^{j+1}m_{1,\frac{n}{2}}
& \cdots & (-1)^{j+1}m_{1,2} & (-1)^{j+1}(m_{1,1}-1)\\
m_{2,1} & m_{2,2} & \cdots & m_{2,\frac{n}{2}} & (-1)^{j+1}m_{2,\frac{n}{2}} & \cdots & (-1)^{j+1}(m_{2,2}-1) & (-1)^{j+1}m_{2,1}\\
\vdots & \vdots &   & \vdots & \vdots &  & \vdots & \vdots\\
m_{\frac{n}{2},1} & m_{\frac{n}{2},2} & \cdots & m_{\frac{n}{2},\frac{n}{2}} & (-1)^{j+1}(m_{\frac{n}{2},\frac{n}{2}}-1) & \cdots &
(-1)^{j+1}m_{\frac{n}{2},2} & (-1)^{j+1}m_{\frac{n}{2},1}\\
m_{\frac{n}{2}+1,1} & m_{\frac{n}{2}+1,2} & \cdots & (-1)^j & 0 & \cdots & (-1)^{j+1}m_{\frac{n}{2}+1,2}
& (-1)^{j+1}m_{\frac{n}{2}+1,1}\\
\vdots & \vdots & \iddots  & \vdots & \vdots & \ddots & \vdots & \vdots\\
m_{n-1,1} & (-1)^j & \cdots  & 0 & 0 &  \cdots & 0 & (-1)^{j+1}m_{n-1,1}\\
(-1)^j & 0 & \cdots  & 0 & 0 & \cdots & 0 & 0
\end{array}
\right) \] }
while, for odd $n$, one simply needs to modify the indices a bit and add the central $\frac{n+1}{2}$-th column
\[ (m_{1,\frac{n+1}{2}}, \cdots , m_{\frac{n-1}{2},\frac{n+1}{2}}, (-1)^j, 0, \cdots, 0) .\]

\section{Matrices and Conjectures}\label{SecMatrices}
We are now going to translate our previous formulas and conjectures in a matrix version which hopefully will make our tasks easier
(at least in small dimensions). We need the matrices associated to all the operators involved in our computations so we fix notations for them 
once and for all and we shall see that everything can be written (basically) in terms of 3 matrices.

\subsection{Atkin Operator}
By the previous section it is easy to see that the matrix associated to $\U_t$ acting on $C_j$ is
\begin{equation}\label{EqAt}
U= M D= M \left(\begin{array}{ccc}
t^{s_1} & \cdots & 0\\
 & \ddots & \\
0 & \cdots & t^{s_n}
\end{array}\right)
\end{equation}
where for $1\leqslant i\leqslant n$ we set $s_i=j+1+(i-1)(q-1)$.

\subsection{Fricke involution} We compute the Fricke action on cocycles.
\begin{align*} 
\c_i^{Fr}(\e)(X^\ell Y^{k-2-\ell}) & = \matrix{0}{-1}{t}{0}^{-1}\c_i\left( \matrix{0}{-1}{t}{0}  \matrix{0}{1}{1}{0} \right) (X^\l Y^{k-2-\l}) \\
 & = \matrix{0}{\frac{1}{t}}{-1}{0} \c_i \left( \matrix{1}{0}{0}{t} \matrix{-1}{0}{0}{1} \right) (X^\l Y^{k-2-\l})\\
  & = (-1)^{k-\l-1} t^{m-\l -1}\c_i(\e)(X^{k-2-\l} Y^\l)
\end{align*}
so that
\begin{equation}\label{EqFrCoc}
\c_i^{Fr}=(-1)^{i+1}t^{i+1+m-k} \c_{k-2-i}
\end{equation}
(note that $\c_i$ and $\c_{k-2-i}$ correspond to ``symmetric'' columns in the block
associated with $C_j$).\\
Therefore, the $b$-th column of the Fricke acting on the block $C_j$ comes from
\[ \c_{j+(b-1)(q-1)}^{Fr}  = t^{m-k} ( (-t)^{j+1+(b-1)(q-1)}\c_{j+(n-b)(q-1)} )\,.\]
Observe that $(-1)^{j+1+(b-1)(q-1)}=(-1)^{j+1}$; so the matrix associated this action is
\begin{equation}\label{MatrixFricke}
t^{m-k}F=t^{m-k}\left( \begin{array}{ccccc}
0 & 0 & \cdots &  0  & (-t)^{s_n} \\
0 & 0 & \cdots &  (-t)^{s_{n-1}} & 0 \\
\vdots & \vdots & \iddots & 0 & \vdots \\
0 & (-t)^{s_2} & 0 & \cdots & \vdots \\
(-t)^{s_1} & 0 & \cdots & \cdots & 0
\end{array} \right) .
\end{equation}

Note that, since $k$ is even when $q$ is odd and $s_i+s_{n-i+1}=k$ for any $i$, we have $F^2=t^kI$ (where $I$ is the identity matrix).
We remark that, letting $A$ be the antidiagonal matrix  
\[ A= \left( \begin{array}{ccc} 0 & \dots  & (-1)^{j+1} \\ 
 & \iddots & \\ (-1)^{j+1} & \dots & 0  \end{array}\right),\] 
one has $AF=D$. As an example of the translations of our previous formulas in matrix form one can easily check that 
\[ (t^{m-k}F)^2=t^{2m-2k}F^2=t^{2m-2k}(-t)^kI=t^{2m-k}I .\]
This corresponds to $(\varphi^{Fr})^{Fr}=t^{2m-k}\varphi$.

\subsection{Trace maps}\label{SecTrMaps}
By equation \eqref{EqTr} we have that the trace action on cocycles is 
\[  Tr(\c_i)= \c_i + t^{-m} \U_t(\c_i^{Fr}),\]
i.e. in terms of matrices
\begin{equation}\label{eqTr} 
T:= I+t^{-m}MD(t^{m-k}F)=I+t^{-k}MAF^2 = I+MA .
\end{equation}


By equation \eqref{EqTr'} (or composing \eqref{eqTr} with the Fricke matrix $t^{m-k}F$) it is easy to see that the matrix for the twisted
trace on $C_j$ is 
\begin{equation}\label{eqTr'}
  T'=t^{m-k}(F+ MD ).
\end{equation}  

Since $Tr(\delta_1(\varphi))=\varphi$ we have $T=T^2$ and $\psi\in Im(\delta_1) \iff Tr(\psi)=\psi$, which yields
\begin{equation}\label{EqT^2}
I+MA=(I+MA)^2=I+2MA+MAMA\quad  {\rm and}\quad Im(\delta_1)=Ker(MA).
\end{equation}
The first relation readily implies
\[ MA(I+MA)=0\quad{\rm and}\quad (I+MA)MA=0 ,\]
i.e. $Im(T)=Ker(MA)=Im(\delta_1)$ (which is obvious) and, since $A$ is invertible,
\[ Im(M)\subseteq Ker(T) .\]
In particular this leads to $Im(\U_t)\subseteq Ker(Tr)$.

Finally there is an obvious relation between $Ker(Tr)$ and $Ker(Tr')$ which, in terms of matrices, reads as
$Ker(T')=F(Ker(T))$  (indeed $\varphi\in Ker(Tr') \iff \varphi^{Fr}\in Ker(Tr)\,$) and we recall that, by \cite[Theorem 3.9]{BV2},
$ Ker(\U_t)=Im(\delta_t)$, i.e. $Ker(MD)=Im(\delta_t)$. Therefore {\em oldforms} are 
\[ S^{1,old}_{k,m}(\Gamma_0(t)):=Ker(MA)\oplus Ker(MD) \] 
(direct sum because $\delta$ is injective) and {\em newforms} are 
\[ S^{1,new}_{k,m}(\Gamma_0(t)):=Ker(T)\cap F(Ker(T))= Ker(I+MA)\cap F(Ker(I+MA)). \]

\subsection{Conjectures II}
By \eqref{EqTrdelta} we have $\varphi\in Ker(\T_t) \iff Tr(\delta_t(\varphi))=0$ and we recall that on forms in $S^1_{k,m}(GL_2(A))$
$\delta_t$ acts as the Fricke map. Therefore $\varphi\in Ker(\T_t)$ yields an element in 
\[ Ker(MA)\cap Ker(T')=Ker(MA)\cap Ker(F+MD)=Ker(MA)\cap F(Ker(I+MA)).\]
Our previous Conjectures \ref{OldConj}, can be now rewritten as

\begin{conjs}\label{NewConj}\ 
\begin{enumerate}
\item[{\bf 1.}] $Ker(MA)\cap Ker(F+MD)=Ker(MA)\cap F(Ker(I+MA))=0$;
\item[{\bf 2.}] $MD$ is diagonalizable if $q$ is odd and, when $q$ is even, it is diagonalizable if and only if 
$\dim_{\C_\infty}S^{1,new}_{k,m}(\Gamma_0(t))\leqslant 1$;
\item[{\bf 3.}] $S^1_{k,m}(\Gamma_0(t))=S^{1,old}_{k,m}(\Gamma_0(t))\oplus S^{1,new}_{k,m}(\Gamma_0(t))=
(Ker(MA)\oplus Ker(MD))\oplus (Ker(T)\cap FKer(T))$.
\end{enumerate}
\end{conjs}

As a starting point we can easily observe that $Ker(MA)\cap Ker(T)=Ker(MD)\cap Ker(T')=0$.

\section{Main theorems and special cases}\label{SecSpecial}
We shall provide a criterion for the conjecture on newforms and oldforms and then use the explicit formula for the matrices
to verify all conjectures for various values of $j$, $n$ and $q$. In particular a few special cases will provide a proof for the conjectures for
cusp forms of weight $k\leqslant 5q-5$, but, with the criterion of Theorem \ref{ThmSum}, it should be quite easy to go much further.

\subsection{Sum of oldforms and newforms}\label{SecSum}
To prove Conjecture {\bf 3} we need 
\[ \left(Ker(MA)\oplus Ker(MD)\right)\cap\left(Ker(I+MA)\cap F(Ker(I+MA))\right)=0\]
\[ S^1_{k,m}(\Gamma_0(t))=\left(Ker(MA)\oplus Ker(MD)\right)+\left(Ker(I+MA)\cap F(Ker(I+MA))\right) .\]
We provide a necessary and sufficient condition for these to hold.

\begin{thm}\label{ThmSum}
We have  
\[  S^1_{k,m}(\Gamma_0(t))= S^{1,old}_{k,m}(\Gamma_0(t))\oplus S^{1,new}_{k,m}(\Gamma_0(t)) \iff
I-t^{-k}(TF)^2\ is\ invertible.\]
\end{thm}

\begin{proof} Assume $I-t^{-k}(TF)^2$ is invertible.
We begin by showing that the intersection between old and newforms is trivial. Let $\eta=\delta(\varphi,\psi)$ be old and new, then
(recall $\varphi,\psi\in S^1_{k,m}(GL_2(A))$ yields $T\varphi=\varphi$ and $T\psi=\psi$,
with a little abuse of notations we denote with the same symbol the modular forms and their associated coordinate vector)
\begin{itemize} 
\item $\eta=\varphi+t^{m-k}F\psi$;
\item $T\eta=T\varphi+t^{m-k}TF\psi=\varphi+t^{m-k}TF\psi=0 \Longrightarrow \varphi= -t^{m-k}TF\psi$;
\item $T'\eta=t^{m-k}(TF\eta)=t^{m-k}(TF\varphi+t^{m-k}TF^2\psi)=0$ implies
\[ 0=t^{m-k}(TF(-t^{m-k}TF\psi)+t^mT\psi)=t^{2m-k}(-t^{-k}(TF)^2\psi+\psi)=t^{2m-k}(-t^{-k}(TF)^2+I)\psi.\]
\end{itemize}
By hypothesis this leads to $\psi=0$, hence $\varphi=0$ and finally $\eta=0$ as well.

\noindent For the sum, given $\Psi\in  S^1_{k,m}(\Gamma_0(t))$ it is enough to find $\varphi,\psi\in  S^1_{k,m}(GL_2(A))$ such that
$\Psi-\delta(\varphi,\psi)$ is new, i.e.
\[ Tr(\Psi-\delta(\varphi,\psi))=Tr(\Psi)-\varphi-Tr(\delta_t(\psi))=Tr(\Psi)-\varphi-Tr(\psi^{Fr})=0 \]
and
\[ Tr'(\Psi-\delta(\varphi,\psi))=Tr'(\Psi)-Tr'(\varphi)-Tr'(\psi^{Fr})=Tr'(\Psi)-Tr'(\varphi)-t^{2m-k}\psi=0.\]
In terms of matrices these read as
\[ \left\{ \begin{array}{l} T\Psi-\varphi-t^{m-k}TF\psi =0 \\
TF\Psi-TF\varphi-t^m\psi=0 \end{array}\right. . \]
Assuming that $I-t^{-k}(TF)^2$ is invertible, we solve for $\varphi$ and $\psi$ getting
\[ \left\{ \begin{array}{l} \varphi=T\Psi-t^{m-k}TF\psi \\
\psi=t^{-m}(TF\Psi-TF(T\Psi-t^{m-k}TF\psi)) = t^{-m}TF(\Psi-T\Psi)+t^{-k}(TF)^2\psi \end{array}\right. \]
\[ \left\{ \begin{array}{l} \psi=(I-t^{-k}(TF)^2)^{-1} t^{-m}TF(\Psi-T\Psi) \\
\varphi=T\Psi-t^{m-k}TF t^{-m}(TF\Psi-TF\varphi)= T\Psi- t^{-k}(TF)^2\Psi+t^{-k}(TF)^2\varphi \end{array}\right. \]
\begin{equation}\label{EqSolve} \left\{ \begin{array}{l} \psi=(I-t^{-k}(TF)^2)^{-1} t^{-m}TF(\Psi-T\Psi) \\
\varphi=(I-t^{-k}(TF)^2)^{-1} (T\Psi - t^{-k}(TF)^2\Psi) \end{array}\right. . \end{equation}
Vice versa let $\eta\neq 0$ be in the kernel of $I-t^{-k}(TF)^2$, so that $TFTF\eta=t^k\eta$, and apply $T$ (recalling $T^2=T$)
to get $TFTF\eta=T^2FTF\eta=t^kT\eta$. This shows $T\eta=\eta$ so $\eta$ is old (and belongs to $Ker(MA)\,$).
Note that $MD\eta\neq 0$, otherwise $0\neq \eta\in Ker(MA)\cap Ker(MD)$: a contradiction
to the injectivity of $\delta$.
Equations \eqref{EqUtdelta1} and \eqref{EqUtdeltat} imply that $MD\eta$ (i.e. ${\bf U}_t(\eta)\,$) is old as well.
Finally
\begin{align*} t^k\eta & = (TF)^2\eta = TF(MD+F)\eta \\
\ & = TF(MD\eta)+TFF\eta = TF(MD\eta) +t^kT\eta  \\
\ & = TF(MD\eta)+t^k\eta .
\end{align*}
Therefore $TF(MD\eta)=0$ and we already noticed in Section \ref{SecTrMaps} that $TM=0$, i.e. $T(MD\eta)=0$ as well.
Hence $MD\eta \in Ker(T)\cap Ker(TF)=S^{1,new}_{k,m}(\Gamma_0(t))$,
\[ 0\neq MD\eta \in S^{1,old}_{k,m}(\Gamma_0(t))\cap S^{1,new}_{k,m}(\Gamma_0(t))\]
and we cannot have a direct sum between them.
\end{proof}

One can easily check that the formulas \eqref{EqSolve} are compatible with the possibility that $\Psi$ is old, i.e.
if $\Psi=\delta_1(\eta)$, then $T\Psi=\Psi$ so in equation \eqref{EqSolve} $\psi=0$ and 
$\varphi=(I-t^{-k}(TF)^2)^{-1} (I - t^{-k}(TF)^2)\eta=\eta$. A similar computation for $\Psi=\delta_t(\eta)=\eta^{Fr}=t^{m-k}F\eta$,
leads to (recall $T\eta=\eta$ and $F^2=t^kI$)
\[ \psi=(I-t^{-k}(TF)^2)^{-1} t^{-m}TF(t^{m-k}F\eta-T(t^{m-k}F\eta))=(I-t^{-k}(TF)^2)^{-1} (I - t^{-k}(TF)^2)\eta=\eta \]
and
\[ \varphi=(I-t^{-k}(TF)^2)^{-1} (T(t^{m-k}F\eta) - t^{-k}(TF)^2(t^{m-k}F\eta))=(I-t^{-k}(TF)^2)^{-1}
t^{m-k}TF(\eta-T\eta)=0.\] 

\begin{rem} 
The condition on the invertibility of the matrix $I-t^{-k}(TF)^2$ is computationally really easy to check. We did it using 
the software Mathematica (\cite{W}). In particular, we checked more than 1200 blocks for $q=2,3,2^2,5,7,2^3,3^2,11$ and  
$0\leqslant j\leqslant q-2$ and $n\leqslant 31$.
\end{rem}

\begin{rem}\label{RemTF}
In the proof of Theorem \ref{ThmSum} we saw that an element in $Ker(I-t^{-k}(TF)^2)$
has to be old. Let $\varphi\in S^1_{k,m}(GL_2(A))$, then
\[  \delta_1\T_t(\varphi)=t^{k-m}\delta_t(\varphi)+ \U_t(\delta_1(\varphi)) = (F+MD)\varphi =TF\varphi \,. \]
Moreover, observe that $I-t^{-k}(TF)^2=(I-t^{-k/2}TF)(I+t^{-k/2}TF)$, so that
$\varphi\in Ker(I-t^{-k}(TF)^2)$ leads to
\[ TF\varphi=-t^{k/2}\varphi\quad{\rm or}\quad TF((I+t^{-k/2}TF)\varphi)=
t^{k/2}(I+t^{-k/2}TF)\varphi.\] 
Therefore, $S^1_{k,m}(\G_0(t))$ is direct sum of oldforms 
and newforms if and only if there does not exist $\eta\in S^1_{k,m}(GL_2(A))$ eigenform of eigenvalue $\pm t^{k/2}$ for $\T_t$.\\
Our computations (see tables at https://sites.google.com/site/mariavalentino84/publications) always provided slopes at level one that are strictly less than $k/2$. Therefore, proving that $k/2$ is un upper bound for slopes of $\T_t$ would immediately
prove also the conjecture on $S^1_{k,m}(\Gamma_0(t))$ being direct sum of new and old forms.
\end{rem}

\subsection{Antidiagonal blocks and newforms}\label{SecAntidiagonalNewforms}
Most of the computations of this section and of the following one will rely on the well known

\begin{lem}\label{KummerThm}(Lucas' Theorem)
Let $n,m\in\N$ with $m\leqslant n$ and write their $p$-adic expansions as
$n=n_0+n_1p+\dots +n_d p^d$, $m=m_0+m_1 p + \dots +m_d p^d$. Then
\[ \binom{n}{m} \equiv \binom{n_0}{m_0} \binom{n_1}{m_1} \dots \binom{n_d}{m_d} \pmod p.\]
\end{lem}

\begin{proof}
See \cite{DW} or \cite{Gr}.
\end{proof}

With notations as in \cite{Cor}, consider the modular form $g\in M_{q-1,0}(GL_2(A))$ and the cusp form
$h\in M_{q+1,1}(GL_2(A))$ which generate $M(GL_2(A))$, i.e., such that
$\bigoplus_{k,m}M_{k,m}(GL_2(A))\simeq \C_\infty [g,h]$ (see \cite[Proposition 4.6.1]{Cor}), where the polynomial ring
is intended doubly graded by weight and type.
When $n\leqslant j+1$ (and still $k=2j+2+(n-1)(q-1)$), using \cite[Proposition 4.3]{Cor} one finds that
the space $M_{k,m}(GL_2(A))$ is zero unless $k=q(q-1)$ and $m=0$ (i.e., $j=q-2$) when it is generated
by the (non-cuspidal) form $g^q$. Moreover, when $n=j+2$ we have
\[ \dim_{\C_\infty} M_{k,m}(GL_2(A)) = \left\{
\begin{array}{ll} 1 & {\rm if}\ j<q-2\ {\rm (generated\ by\ } h^{j+1}{\rm )} \\
2 & {\rm if}\ j=q-2\ {\rm (generated\ by\ } \{h^{q-1},g^{q+1}\}{\rm )} \end{array}\right. . \]
Hence for $n\leqslant j+1$ we do not have oldforms and Conjecture {\bf 1} is trivial. For Conjecture {\bf 3} we have to prove that
all forms in $S^1_{k,m}(\Gamma_0(t))$ are new (obviously they cannot arise in any way from forms of lower level but we have to check
that they are new according to our Definition \ref{DefNewforms}).

\begin{thm}\label{ThmAntidiagonal}
Let $n\in\N$ and $0\leqslant j\leqslant q-2$. Then, for all $n\leqslant j+1$, the matrix $M=M(j,n,q)$ is antidiagonal
(since now $j$, $n$ and $q$ may vary we often put this 3 parameter explicitly in the notation for the matrix $M$).
\end{thm}

\begin{proof}
Thanks to the symmetries of the matrices $M(j,n,q)$, we simply need to check the general
$b$-th column for $1\leqslant b \leqslant \frac{n}{2}$ (or $\leqslant \frac{n+1}{2}$ according to the parity of $n$)
and above the antidiagonal (i.e. for $b<n+1-a$).\\
Let us start with the elements on the diagonal. We rewrite them as
\[  m_{a,a}=(-1)^{j}\binom{(n-a)q +j +a-n}{(a-1)q+j-a+1}\,.  \]
Our hypotheses on $j$ and $n$ yield $0\leqslant j+a-n,j-a+1 < q=p^r$, hence, in order to use Lemma \ref{KummerThm}, we can write the
$p$-adic expansion of the terms in the binomial coefficient as
\begin{align*}
(n-a)q+j+a-n & =\a_0+\a_1 p+\dots+\a_{r-1}p^{r-1}+(n-a)p^r; \\
(a-1)q+j-a+1 & =\delta_0+\delta_1 p+\dots+\delta_{r-1}p^{r-1}+(a-1)p^r.
\end{align*}
If there exists $i$ such that $\a_i<\delta_i$, then $\binom{\a_i}{\delta_i}=0$ and $m_{a,a}$ is zero. Otherwise, if
for any $i$ one has $\a_i\geqslant\delta_i$, then $j+a-n\geqslant j+1-a$, i.e. $a-1\geqslant n-a$. Again we get $m_{a,a}=0$
unless $a=\frac{n+1}{2}$ where we have already seen that $m_{\frac{n+1}{2},\frac{n+1}{2}}=(-1)^{j}$.

The other $m_{a,b}$ are
\[ m_{a,b}= -\left[ \binom{(n-a)q+j+a-n}{(n-b)q+j+b-n} + (-1)^{j+1}\binom{(n-a)q+j+a-n}{(b-1)q+j-b+1}\right] \,.\]
As before, $j+a-n,j-b+1,j+b-n < q$ and all of them are non-negative. Thus, the $p$-adic expansions
of the terms involved in the coefficients are
\begin{align*} (n-a)q+j+a-n & =\a_0+\a_1 p+\dots+\a_{r-1}p^{r-1}+(n-a)p^r; \\
(b-1)q+j-b+1&=\b_0+\b_1 p+\dots+\b_{r-1}p^{r-1}+(b-1)p^r; \\
(n-b)q+j+b-n&=\g_0+\g_1 p+\dots+\g_{r-1}p^{r-1}+(n-b)p^r.
\end{align*}
By Lemma \ref{KummerThm} we have that
\begin{itemize}
\item if there exists $i$ such that $\a_i<\g_i$, then the first binomial coefficient is zero;
\item if there exists $i$ such that $\a_i<\b_i$, then the second binomial coefficient is zero.
\end{itemize}
Otherwise, if for any $i$ one has $\a_i\geqslant\g_i$, then $j+a-n\geqslant j+b-n$ and $n-b\geqslant n-a$.
This implies that the first binomial coefficient is zero
(observe that the equality $a=b$ cannot happen here, the entry $m_{a,a}$ has already been treated).\\
For the second coefficient assume that for any $i$ one has $\a_i\geqslant \b_i$, then
$j+a-n\geqslant j-b+1$, i.e., $b\geqslant n+1-a$, a contradiction to our assumption of being above the antidiagonal.\end{proof}

\begin{cor}\label{CorDiagAntidiag}
Let $n\in\N$ and $0\leqslant j\leqslant q-2$. Then, for all $2\leqslant n\leqslant j+1$, the matrix associated with $\U_t$, i.e.
$MD:=M(j,n,q,t)$, is diagonalizable if and only if $q$ is odd.
\end{cor}

\begin{exe} \label{Exn=j+2o3}
{\em The following matrices show that the bound in Theorem \ref{ThmAntidiagonal} is sharp, i.e. the appearance of
oldforms causes a non-antidiagonal action. For $q=8$, $j=3, 6$ we have }
\[ M(3,5,8)=\left( \begin{array}{ccccc}
1 & 0 & 0 & 0 & 0 \\
0 & 0 & 0 & 1 & 0 \\
0 & 0 & 1 & 0 & 0 \\
0 & 1 & 0 & 0 & 0 \\
1 & 0 & 0 & 0 & 0 \end{array}\right) \ {\rm and} \
M(6,8,8)=\left(\begin{array}{cccccccc} 1 & 1 & 1 & 1 & 1 & 1 & 1 & 0\\
0 & 0 & 0 & 0 & 0 & 0 & 1 & 0\\
0 & 0 & 0 & 0 & 0 & 1 & 0 & 0\\
0 & 0 & 0 & 0 & 1 & 0 & 0 & 0\\
0 & 0 & 0 & 1 & 0 & 0 & 0 & 0\\
0 & 0 & 1 & 0 & 0 & 0 & 0 & 0\\
0 & 1 & 0 & 0 & 0 & 0 & 0 & 0\\
1 & 0 & 0 & 0 & 0 & 0 & 0 & 0 \end{array}\right).\]
\end{exe}

\begin{thm}\label{AntidiagAreNew}
If $n\leqslant j+1$ all Conjectures \ref{NewConj} hold.
\end{thm}

\begin{proof}
We already mentioned that Conjectures {\bf 1} and {\bf 2} hold (trivially or by Corollary \ref{CorDiagAntidiag}).
By Theorem \ref{ThmAntidiagonal} we know that $M=M(j,n,q)$ is antidiagonal. In particular
\[  M=\left( \begin{array}{ccc}
0 & \cdots & (-1)^j\\
 & \iddots & \\
(-1)^j & \cdots & 0
\end{array}\right) \]
Hence $MA=-I$ and $MD=-F$, i.e. $T=I+MA$ and $T'=F+MD$ are both the null matrix and 
$S^1_{k,m}(\Gamma_0(t))=S^{1,new}_{k,m}(\Gamma_0(t))$ (note that, by Theorem \ref{ThmSum}, this provides another
proof of the fact that $S^{1,old}_{k,m}(\Gamma_0(t))=0$).
\end{proof}

It seems relevant to notice that all the eigenforms involved in an antidiagonal block are newforms (i.e. this holds even if the whole matrix 
is not antidiagonal). Indeed the existence of an antidiagonal block yields equations like
\[ U_t(\c_{j+(h-1)(q-1)})= (-1)^j t^{j+1+(h-1)(q-1)}\c_{k-2-j-(h-1)(q-1)} \]
(for the cocycles involved in the block) and we recall equation \eqref{EqFrCoc}
\[ \c_i^{Fr}=(-1)^{i+1}t^{i+1+m-k} \c_{k-2-i} .\]
Substituting in equations \eqref{EqTr} one gets
\begin{align*}
Tr(\c_{j+(h-1)(q-1)}) & =\c_{j+(h-1)(q-1)} + t^{-m}\U_t(\c_{j+(h-1)(q-1)}^{Fr}) \\
\ & = \c_{j+(h-1)(q-1)} + (-1)^{j+1+(h-1)(q-1)} t^{j+1+(h-1)(q-1)+m-k-m}\U_t(\c_{j+(n-h)(q-1)}) \\
\ & = \c_{j+(h-1)(q-1)} + (-1)^{j+1} t^{j+1+(h-1)(q-1)-k}(-1)^j t^{j+1+(n-h)(q-1)}\c_{j+(h-1)(q-1)}  = 0 
\end{align*} 
(where the last equality follows also from $k=2j+2+(n-1)(q-1)\,$). The computations to show
$Tr'(\c_{j+(h-1)(q-1)})=0$ are similar (substituting in $\eqref{EqTr'}$).

\subsection{Three more cases: $j=0$, $n=j+2$ and $n\leqslant 4$}\label{Cor-j=0&n=j+2}
We briefly describe a few more cases in which our Theorem \ref{ThmSum} and the particular form of the matrices lead to
a proof of all the conjectures.

\begin{thm}\label{Thmj=0}
Let $n\in\N$ with $n\geqslant 2$ and $j=0$. Then, for all $n\leqslant q+2$,
the matrix $M(0,n,q)$ has the following entries
\begin{enumerate}
\item {$m_{a,1}=1$ for $1\leqslant a\leqslant n$;}
\item {$m_{a,b}=0$ for $1\leqslant a\leqslant n-2$, $2\leqslant b\leqslant \frac{n}{2}$ (or $\frac{n+1}{2}$
depending on the parity of $n$) and $b<n+1-a$,}
\end{enumerate}
i.e.,
\[ M(0,n,q)=\left( \begin{array}{cccccc}
1 & 0 & \cdots & \cdots & 0 & 0 \\
1 & 0 & \cdots & 0 & 1 & -1 \\
\vdots & \vdots & \ & \udots & 0 & \vdots \\
\vdots & 0 & \udots &  & \vdots & \vdots \\
1 & 1 & 0 & \cdots & 0 & -1 \\
1 & 0 & \cdots & \cdots & 0 & 0 \end{array}\right)\,.\]
\end{thm}

\begin{proof}
We need just to apply repeatedly Lemma \ref{KummerThm} as done in the proof of Proposition \ref{ThmAntidiagonal}.
\end{proof}

\begin{exe}
As before we can show that the bound on $n$ is the best possible: indeed
\[ M(0,6,3)=\left( \begin{array}{cccccc}
1 & 0 & 0 & 0 & 0 & 0\\
1 & 1 & 0 & 0 & 0 & -1\\
1 & 0 & 0 & 1 & 0 & -1\\
1 & 0 & 1 & 0 & 0 & -1\\
1 & 1 & 0 & 0 & 0 & -1\\
1 & 0 & 0 & 0 & 0 & 0
\end{array}\right)\]

\end{exe}

\begin{cor}\label{Corj=0}
With hypotheses as in Theorem \ref{Thmj=0} we have that all Conjectures \ref{NewConj} hold.
\end{cor}

\begin{proof}
We have
\[ T=\left( \begin{array}{ccccc}
1 & 0 & \cdots &  0 & -1 \\
\vdots & \vdots &  & \vdots & \vdots \\
1 & 0 & \cdots & 0 & -1 \\
0 & 0 & \cdots & 0 & 0 \end{array}\right)\quad{\rm and} \quad
TF=\left( \begin{array}{ccccc}
t^{s_1} & 0 & \cdots &  0 & -t^{s_n} \\
\vdots & \vdots &  & \vdots & \vdots \\
t^{s_1} & 0 & \cdots & 0 & -t^{s_n} \\
0 & 0 & \cdots & 0 & 0 \end{array}\right) .\]
So it is easy to see that $I-t^{-k}(TF)^2=(I-t^{-\frac{k}{2}}TF)(I+t^{-\frac{k}{2}}TF)$ is invertible. 
For completeness we mention that oldforms are spanned by
\[ \langle \c_0+\cdots+\c_{(n-2)(q-1)}, t^{n-2}\c_{q-1}+t^{n-3}\c_{2(q-1)}+\cdots+t \c_{(n-2)(q-1)}+\c_{(n-1)(q-1)} \rangle, \]
while newforms are generated by
\[  \langle \c_{q-1}, \cdots, \c_{(n-2)(q-1)}\rangle\,. \] 
For the injectivity of $\T_t$, direct computation show  
\[  Ker(MA)=< \c_0+\c_{q-1}+\cdots+\c_{(n-2)(q-1)}> \]
and 
\[  Ker(F+MD)= < t^{(n-1)(q-1)}\c_0+\c_{(n-1)(q-1)}, \c_{q-1}, \cdots, \c_{(n-2)(q-1)}>\,,  \]
so their intersection is trivial.\\
Finally, diagonalizability (or non diagonalizability) follows from the central antidiagonal block and the calculation of 
the characteristic polynomial (note that $k\geqslant 3$ in our range here)
\[ \det(M(0,n,q) D-XI)=\left\{ \begin{array}{ll} (X^2-tX)(X^2-t^k)^{\frac{n}{2}-1} & {\rm if}\ n\ {\rm is\ even} \\
\ & \\
(X^2-tX)(X^2-t^k)^{\frac{n-3}{2}}(-X+t^{\frac{k}{2}}) & {\rm if}\ n\ {\rm is\ odd} \end{array} \right. \,.\qedhere\]
\end{proof}

\begin{thm}\label{Thmn=j+2}
If $n=j+2$ with $0\leqslant j\leqslant q-2$ the matrix $M(j,j+2,q)$ has the following shape:
\[ M=\left( \begin{array}{cccccc}
1 & m_{1,2} & \cdots & \cdots & (-1)^{j+1}m_{1,2} & 0 \\
0 & 0 & \cdots & 0 & (-1)^j & \vdots \\
\vdots & \vdots & \ & \udots & 0 & \vdots \\
\vdots & 0 & \udots &  & \vdots & \vdots \\
0 & (-1)^j & 0 & \cdots & 0 & \vdots \\
(-1)^j & 0 & \cdots & \cdots & 0 & 0 \end{array}\right)\,. \]
\end{thm}

\begin{proof}
Apply again Lemma \ref{KummerThm} as already done in the proofs of Theorems \ref{Thmj=0} and \ref{ThmAntidiagonal}.
\end{proof}

\begin{cor}\label{Corn=j+2}
With hypotheses as in Theorem \ref{Thmn=j+2} we have that all Conjectures \ref{NewConj} hold.
\end{cor}

\begin{proof}
Using
\[ T= \left( \begin{array}{ccccc}
1 & m_{1,2} & \cdots & (-1)^{j+1}m_{1,2} & (-1)^{j+1} \\
0 & 0 & \cdots & 0 & 0\\
\vdots & \vdots & & \vdots & \vdots \\
0 & \cdots & \cdots & \cdots & 0 \end{array}\right) \quad {\rm and}\quad
TF=\left(\begin{array}{ccccc}
t^{s_1} & m_{1,2}t^{s_2} & \cdots & m_{1,2}t^{s_{n-1}} & t^{s_n}\\
0 & 0 & \cdots & 0 & 0 \\
\vdots & \vdots & \ddots & \vdots & \vdots\\
0 & 0 & \cdots & 0 & 0
\end{array}\right) \]
computations are as in the previous corollary (and even easier). Oldforms are generated by 
\[  Ker(MA)=\langle \c_j\rangle \quad {\rm and} \quad  Ker(MD)=\langle \c_{j+(n-1)(q-1)}\rangle \]
and newforms are spanned by $\langle \c_{j+(q-1)}, \cdots, \c_{j+(n-2)(q-1)} \rangle$, no matter the values of the $m_{1,b}$.
The characteristic polynomial is
\[ \det (M(j,j+2,q)D-XI)= (X^2-t^{j+1}X)\cdot \left\{ \begin{array}{ll}
(X^2-t^k)^{\frac{n-2}{2}} & {\rm if\ }n\ {\rm is\ even} \\
\ & \\
(X^2-t^k)^{\frac{n-3}{2}}(-X+(-1)^j t^{\frac{k}{2}}) & {\rm if\ }n\ {\rm is\ odd} \end{array}\right. ,\]
and diagonalizability is straightforward (even without an antidiagonal block).
\end{proof}

In low dimension (i.e. for small $k$) the previous theorems prove the conjectures for $n\leqslant 3$ and we only need one more matrix
to check them for $n\leqslant 4$ (i.e. for $k\leqslant 2(q-2)+2+3(q-1)=5q-5$). We provide this final example for completeness; since we 
only have to consider the cases $n\geqslant j+3$, it should be easy  to go on with explicit computations for small $n$.  
 
\begin{thm}
If $n\leqslant 4$ all Conjectures \ref{NewConj} are true.
\end{thm}

\begin{proof}
As mentioned above we only need to check $n=4$ for $j=1$: we have that $k=3q+1$  
\[   
M=\left(\begin{array}{cccc}
2 & -2 & -2 & 1\\
1 & -1 & -2 & 1\\
0 & -1 & 0 & 0\\
-1 & 0 & 0 & 0
\end{array}\right)\, , \quad {\rm and} \quad
TF=\left(\begin{array}{cccc}
2t^2 & -2t^{q+1} & -2t^{2q} & 2t^{3q-1}\\
t^2 & -t^{q+1} & -t^{2q} & t^{3q-1}\\
0 & 0 & 0 & 0\\
0 & 0 & 0 & 0
\end{array}\right). \]
The matrix $I-t^{-k}(TF)^2=(I-t^{-\frac{k}{2}}TF)(I+t^{-\frac{k}{2}}TF)$ is invertible and the characteristic polynomial for
$MD$ is $P(X)=X(X+t^{q+1}-2t^2)(X^2-t^{3q+1})$ so all Conjectures hold. In particular newforms are generated by
\[ \langle (t^{2q}-t^{q+1})\c_1+(t^{2q}-t^2)\c_{q}+(t^2-t^{q+1})\c_{2q-1},(t^{q+1}-t^{3q-1})\c_1
+(t^2-t^{3q-1})\c_{q}+(t^2-t^{q+1})\c_{3q-2}\rangle \]
and oldforms by
\[ \langle 2\c_1+\c_q,t^{q-1}\c_{2q-1}+2\c_{3q-2}\rangle .\qedhere \]
\end{proof}

\section{Bounds on slopes}\label{SecBounds}
Since we know that all newforms have slope $k/2$ and we believe that 
$S^1_{k,m}(\G_0(t))$ is the direct sum of oldforms and newforms, we only need to bound
slopes of oldforms. We do this by looking at the Newton polygon of the characteristic
polynomial for $\U_t$, obtaining a sharp lower bound for the slopes and upper bounds
for both slopes and their multiplicities (i.e. the dimension of the space generated
by all eigenforms of a given slope). We recall that our data indicate that $k/2$ is
a sharp upper bound (and Remark \ref{RemTF} strengthens this belief), unfortunately
the actual bound of Theorem \ref{ThmUpperBSlopes} is still quite far from it (while the
bound for multiplicities obtained in Theorem \ref{ThmBoundDim} is optimal and analogue
to the one of \cite{Bu} for the characteristic zero case).

Thanks to \cite[Theorem 3.9]{BV2} we have that \[ r := \dim_{\C_\infty} S^1_{k,m}(GL_2(A))=  \dim_{\C_\infty} Ker(\U_t)\,.\]
Therefore, the characteristic polynomial of $\U_t$ on $C_j$ looks like 
\[  P_{\U_t}(X)=X^n +\l_1X^{n-1}+ \cdots + \l_{n-r}X^{n-(n-r)} \]
 with $\ell_{n-r}\neq 0$. Looking at the form of our matrix $MD$ (in particular the fact that the $i$-th column is divisible exactly by
$t^{s_i}$), we have that 
\[ \l_i= \sum_{\begin{subarray}{c} 1\leqslant \iota_1,\dots,\iota_i\leqslant n \\
\iota_1<\iota_2<\cdots <\iota_i \end{subarray}}  \ell_{\iota_1,\cdots,\iota_i} t^{s_{\iota_1}+\cdots +s_{\iota_i}} \] 
for suitable $\ell_{\iota_1,\cdots,\iota_i}\in \F_q$ and $\ell_0=1$. Let $Q_i=(i,v_t(\l_i))$ for $i=0,\cdots, n-r$ be the points of the 
Newton polygon associated to $p(x)$. Of course, if $\l_i=0$ we skip the corresponding $Q_i$.
We are looking for bounds on $v_t(\l_i)$, hence on slopes and their multiplicity by \cite[Ch IV, Lemma 4]{Ko} 
\footnote{There is quite a difference between our notations and the one in \cite[Ch IV, Lemma 4]{Ko}, but we could not find a more suitable 
reference and, in our opinion, our computations are clearer with our notations.} which here reads as

\begin{lem}\label{LemmaKo}
Let $\alpha\in \mathbb{Q}$. We say that $\alpha$ is a slope of multiplicity $d(\alpha,k)$ for $\U_t$ if there are exactly
$d(\alpha,k)$ roots of $P_{\U_t}(X)$ having $t$-adic valuation $\alpha$. If the Newton polygon of the polynomial $P_{\U_t}(X)$ has a segment
of slope $\alpha$ and projected length $d(\alpha,k)$ (i.e. the length of the projection of the segment on the $x$ axis), 
then $\alpha$ is a slope of multiplicity $d(\alpha,k)$ for $\U_t$. 
\end{lem} 

In order to do this, first observe that since the $s_i=j+1+(i-1)(q-1)$ are all distinct, the sums 
$s_{\iota_1}+\cdots +s_{\iota_i}$ are all distinct too. Moreover, the $s_i$ are increasing. Then we have:
\begin{align}\label{LowerBound}
v_t(\l_i) & \geqslant \min\{ s_{\iota_1}+\cdots+s_{\iota_i} \} = s_1+\cdots+s_i \\ \nonumber
 & = \sum_{h=0}^{i-1} (j+1+h(q-1)) = i(j+1) +\frac{i(i-1)}{2}(q-1)
\end{align}
and
\begin{align}\label{UpperBound}
v_t(\l_i) & \leqslant \max\{ s_{\iota_1}+\cdots+s_{\iota_i} \} = s_n +s_{n-1}+\cdots + s_{n-i+1} \\ \nonumber
 & = \sum_{h=n-i}^{n-1} (j+1 +h(q-1)) = i(j+1) +\left(in - \frac{i(i+1)}{2}\right)(q-1).
\end{align}
Using the above bounds we can plot the points 
\begin{align*}
P_i & = \left( i, i(j+1) +\frac{i(i-1)}{2}(q-1) \right)\\
R_i & = \left( i, i(j+1) +\left(in - \frac{i(i+1)}{2}\right)(q-1) \right)
\end{align*}
and the Newton Polygon of $P_{\U_t}(X)$ lies on or above the $P_i$'s. Looking at the segment joining $P_0=(0,0)$ and $P_1=(1,j+1)$
we immediately have

\begin{prop}\label{SmallestSlope}
The smallest possible slope for $\U_t$ is $j+1$, moreover its multiplicity is 
\[  d(j+1,k) \leqslant 1\,. \]
\end{prop}

\begin{rem}
The above result was already known for cusp forms with $A$-expansion: see \cite[Theorem 2.6]{Pe}. Moreover
if one considers the action of $\U_t$ on the whole $S^1_k(\Gamma_1(t))$ one finds $d(j+1,k) = 1$ as mentioned in
\cite[Lemma 2.4]{Ha} (which uses a different normalization so that our $j+1$ becomes 0). This also shows that
our eigenforms of slope 1 should play the role of classical ``ordinary'' forms (or of a renormalization of them): 
for a completely different and more geometric approach see also \cite{NR}.
\end{rem}

Using the bounds in formulas \eqref{LowerBound} and \eqref{UpperBound} we can prove also the following.

\begin{thm} \label{ThmUpperBSlopes}
If $\a$ is a slope for $\U_t$ of multiplicity $d(\a,k):=d\geqslant 1$, then 
\[  \a\leqslant j+1 +[(n-r)n-1](q-1)\,. \]
\end{thm}

\begin{proof}
Let $\a$ be a slope with multiplicity $d\geqslant 1$. Then, there exists an $i$ such that  the segment connecting $Q_i$ and $Q_{i+d}$ 
has slope $\a$. Note that in particular: $i\geqslant 0$, $i+d\leqslant n-r$ and $1\leqslant d\leqslant n-r$.
By hypothesis $\a d = v_t(\l_{i+d})-v_t(\l_i) $, thus
\[  \min\{ v_t(\l_{i+d})\} -\max\{v_t(\l_i) \} \leqslant \a d \leqslant \max\{ v_t(\l_{i+d})\} - \min\{ v_t(\l_{i})\} \,.\]
By the right inequality we find:
\begin{align*}
\a d &  \leqslant (i+d)(j+1) +\left[  (i+d)n -\frac{(i+d)(i+d+1)}{2} \right](q-1)  -i(j+1) -\frac{i(i-1)}{2}(q-1) \\
 & = d(j+1) +\left[ (i+d)n-\frac{2i^2+2id+d^2+d}{2} \right](q-1) \\
 & = d(j+1) +\left[ (i+d)(n-i)-\frac{d(d+1)}{2} \right](q-1)\,.
\end{align*}
Dividing by $d$ and using the above bounds, we get
\begin{align*}
\a & \leqslant j+1+\left[ \frac{(i+d)(n-i)}{d} - \frac{d+1}{2} \right](q-1)\\
 & \leqslant j+1 +[(n-r)n-1](q-1)\,.\qedhere
\end{align*}
\end{proof}

After estimating the slope we can estimate the multiplicity as well. 

\begin{thm}\label{ThmBoundDim}
Let $\alpha\in \mathbb{Q}$, then
\[  d(\a,k)\leqslant 2 \left( \frac{\a-j-1}{q-1}\right)+1\,. \]
\end{thm}

\begin{proof}
First we observe that the convexity of the Newton Polygon ensures that the slopes are increasingly ordered. 
So, to obtain the maximal value for $d(\a,k)$ we draw the line from $Q_0=(0,0)$ of slope $\a$ and find 
its intersection with the plot of the points $P_i$, i.e. we find the maximal index $i$ such that the line is still above the point $P_i$.
That $i$ represents an upper bound for $d(\a,k)$ by Lemma \ref{LemmaKo}. 
We have
\[ \a i\geqslant i(j+1)+\frac{i(i-1)}{2}(q-1)\,.\]
Then
\[ \frac{i+1}{2}\leqslant \frac{\a-j-1}{q-1}\
\Longrightarrow \ i\leqslant 2\left(\frac{\a-j-1}{q-1}\right)+1\]
and the claim follows.
\end{proof}

\begin{rem}
For $\a\leqslant j+1$ we obtain the result already indicated by Proposition \ref{SmallestSlope}.
\end{rem}

\subsection{Further conjectures}
Looking at our data, in \cite[Section 5]{BV2} we conjectured

\begin{conj}\label{GMConj}
If $k_1,k_2\in \Z$ are both $>2\alpha$ and $k_1\equiv k_2\pmod{(q-1)q^{n-1}}$ for some $n\geqslant \alpha$, then 
$d(k_1,\alpha)=d(k_2,\alpha)$. 
\end{conj}

In \cite[Theorem 2.10]{Ha} the author, using properties of the matrix for $\U_t$ (which he defines {\em glissando matrix}) proves

\begin{thm}[Hattori]
Let $k$ and $n$ be integers satisfying $k\geqslant 2$ and $n\geqslant 0$. Let $\alpha$ be a non-negative rational number 
satisfying $\alpha\leqslant n$ and $\alpha<k-1$. Then we have $d(k+p^n,\alpha)=d(k,\alpha)$.
\end{thm}

A closer look at the data for $\T_t$ acting on $S^1_{k,m}(GL_2(A))$ (i.e. focusing on ``old'' slopes) hints at the following 
refinement of Conjecture \ref{GMConj}. 

\begin{conj}\label{ConjGMRef}
Let the type $m$ be fixed. For any weight $k$, let $\ell(k)\in \N$ be the smallest integer such that $q^{\ell(k)}+2\geqslant k$. 
Then at weight $k':=k+(q-1)q^{\ell(k)}$ (in level 1) we find:\begin{itemize}
\item[{\bf 1.}] the old slopes at weight $k$ with exactly the same multiplicity, i.e. for any old slope $\alpha$ in weight $k$
we have $d(\alpha,k')=d(\alpha,k)$;
\item[{\bf 2.}] the slope $\frac{k}{2}$ with $d(k',\frac{k}{2})=\dim_{\C_\infty}S^{1,new}_{k,m}(\Gamma_0(t))$ 
(note that in weight $k'$ the slope $\frac{k}{2}$ is old and our previous results/conjectures predict that it is not present
among the old slopes at weight $k$). 
\end{itemize}
\end{conj}

\noindent In general the slopes predicted by Conjecture \ref{ConjGMRef} do not describe all slopes at weight $k'$, nevertheless
the conjecture gives support to the existence of families of cusp forms and predicts where to look for them.

\begin{exe}
{\em At the web page https://sites.google.com/site/mariavalentino84/publications look at the file ``Slopes\_Tt\_q2.pdf'' for slopes for
$\T_t$ acting on $M_k(GL_2(A))$ when $q=2$. Since we are interested in cusp forms only, just ignore the largest slope in each weight 
because that one is related to the only form in the basis which is not cuspidal.

\noindent Let $k_0=5$. Then $\ell(5)=2$ and at weight $k_1=k_0+2^2=9$ we find slopes $5/2,5/2,1$. It is easy to see (e.g.
from the file ``CharPoly\_Ut\_Gamma1.pdf''), that $\dim_{\C_\infty}S^{1,new}_{5}(\Gamma_0(t))=2$.\\
Iterating, $\ell(k_1)=3$ and at weight $k_2=k_1+2^3=17$ we find slopes $9/2,9/2,5/2,5/2,1$; we also have that 
$\dim_{\C_\infty}S^{1,new}_{9}(\Gamma_0(t))=2$. Moving on note that $\ell(k_2)=4$, so we find: 
\begin{itemize}
\item[$k_3=33$] with slopes $\{17/2, 17/2, 17/2, 17/2, 17/2, 17/2, 9/2, 9/2, 5/2, 5/2, 1\} $ and $\ell(k_3)=5$;
\item[$k_4=65$] with slopes $\{33/2, 33/2, 33/2, 33/2, 33/2, 33/2, 33/2, 33/2, 33/2, 33/2, 17/2, 17/2, 17/2, 17/2, 17/2,$ \\ 
$ 17/2, 9/2, 9/2, 5/2, 5/2, 1\}$ and $\ell(k_4)=6$;
\item[$k_5=129$] with slopes  $\{65/2, 65/2, 65/2, 65/2, 65/2, 65/2, 65/2, 65/2, 65/2, 65/2, 65/2, 65/2, 65/2, 65/2, 65/2, $\\
$ 65/2, 65/2, 65/2,  65/2, 65/2, 65/2, 65/2, 33/2, 33/2, 33/2, 33/2, 33/2, 33/2, 33/2, 33/2, 33/2, 33/2,   $\\
$17/2, 17/2, 17/2, 17/2, 17/2,  17/2, 9/2, 9/2, 5/2, 5/2, 1\}$.
\end{itemize}
Finally we observe that $\dim_{\C_\infty}S^{1,new}_{17}(\Gamma_0(t))=6$, $\dim_{\C_\infty}S^{1,new}_{33}(\Gamma_0(t))=10$, 
$\dim_{\C_\infty}S^{1,new}_{65}(\Gamma_0(t))=22$.\\
Similar examples can be obtained starting from a different $k_0$, or cosidering different $q$ (odd or even) and looking at the other tables. \\
For more data one can also see Hattori's Tables \cite{Ha1}.}
\end{exe}

\begin{rem}
The exponent $\ell(k)$ in Conjecture \ref{ConjGMRef} seems to be optimal. Indeed, 
in the same setting of the previous example, consider $k_0=11$ for which we find slopes $\{5, 3, 1\}$. 
Then, $\ell(k_0)=4$ and at weight $k_1=11+2^{\ell(k_0)}=27$ we find slopes $\{ 13, 11, 11/2, 11/2, 11/2, 11/2, 5, 3, 1\}$.   
At weight $19=11+2^{\ell(k_0)-1}$ we find slopes $\{9, 11/2, 11/2, 5, 3, 1\}$, but $11/2$ does not show up with the predicted 
multiplicity, indeed $\dim_{\C_\infty}S^{1,new}_{11}(\Gamma_0(t))=4$. Another example: with $m=0$ and $q=3$ (file 
 {\rm ``Slopes\_Tt\_q3\_type0.pdf''}) take $k_0=8$ with $\ell(k_0)=2$; we have the slope $4$ with multiplicity $1$ at weight 
$k_1=k_0+2\cdot 3^{\ell(k_0)}=26$, while the slope $4$ does not appear in weight $k_0+2\cdot 3^{\ell(k_0)-1}=14$. Then the slope
$\frac{k_1}{2}=13$ appears with multiplicity $6$ at weight $80=k_1+2\cdot 3^{\ell(k_1)}$ and is not present at weight $44$
(it appears at weights $38=26+2^2\cdot 3$ and $62=26+2^2\cdot 3^2$ but with the ``wrong'' multiplicity $2$).
\end{rem}

\end{document}